\documentclass[12pt,reqno]{amsart}
\usepackage{geometry}
\usepackage{amssymb}
\usepackage{amsmath}
\usepackage{color}
\usepackage[all]{xy}
\usepackage[OT2,T1]{fontenc}
\usepackage{bm}
\usepackage{thmtools, thm-restate}
\usepackage{pdfsync}
\usepackage{hyperref}
\usepackage{mathrsfs}
\usepackage{dsfont}
\usepackage{textcomp}
\usepackage{tipa}
\usepackage{mathtools}
\usepackage[english]{babel}
\usepackage[autostyle]{csquotes}
\usepackage{xcolor}
\usepackage{enumerate}
\usepackage{upgreek}
\usepackage{bbm}
\usepackage{tikz}
\usepackage{bbm}
\usepackage[T1]{fontenc}

\usepackage{amscd}

\usepackage{caption,booktabs}
\usepackage{outlines}
\usepackage[flushleft]{threeparttable}
\usepackage{cleveref}

\usepackage[backend=biber,
style=alphabetic,]{biblatex}
\title{A bibLaTeX example}
\newcommand{\Mod}[1]{\ (\mathrm{mod}\ #1)}

\newcommand{\Q}{\mathbb{Q}}

\newcommand{\Z}{\mathbb{Z}}

\newcommand{\F}{\mathbb{F}}
\newcommand{\PP}{\mathbb{P}}

\newcommand{\BF}{\mathbb{F}}

\newcommand{\CL}{\mathcal{L}}
\newcommand{\CM}{\mathcal{M}}
\newcommand{\CH}{\mathcal{H}}
\newcommand{\CS}{\mathcal{S}}
\newcommand{\CC}{\mathcal{C}}

\DeclareMathOperator{\tr}{tr}

\DeclareMathOperator{\rk}{rk}

\newtheorem{lem}{Lemma}[subsection]

\newtheorem{prop}[lem]{Proposition}
\newtheorem{cor}[lem]{Corollary}

\newtheorem{claim*}{Claim}
\newtheorem{thm}[lem]{Theorem}

\newtheorem{rem}[lem]{Remark}

\title{Average Rank of Elliptic Curves Over Function Fields}

\author{{\i}rmak Bal\c{c}{\i}k}
\address{Northwestern University, Department of Mathematics, 2033 Sheridan Road, Evanston, IL 60208, USA}
\email{irmak.balcik@northwestern.edu}

\begin{document}

\begin{abstract}
    Let $q$ be a prime with $q \geq 5$. We show that the average rank of elliptic curves over a function field $\BF_{q}(t)$, when ordered by naive height, is bounded above by $25/14 \approx 1.8$. Our result improves the previous upper bound of $2.3$ proven by Brumer in \cite{Brumer}. The upper bound obtained is less than $2$, which shows that a positive proportion of elliptic curves has either rank $0$ or $1$. The proof adapts the work of Young, which shows in \cite{Matthew} that under the assumption of the General Riemann Hypothesis for $L$-functions of elliptic curves, the average rank for the family of elliptic curves over the rational numbers is bounded above by $ 25/14 \approx 1.8$. 
\end{abstract}

\maketitle

\section{Introduction}
Let $E$ be an elliptic curve over a global field $K$. The Mordell-Weil theorem asserts that the group of $K$-rational points $E(K)$ of $E$ forms a finitely generated group $E(K) \cong E(K)_{\text{tor}} \oplus \mathbb{Z}^{r}$ where $E(K)_{\text{tor}}$ is the finite torsion subgroup and $r \in \mathbb{Z}^{\geq 0}$ is the rank. The study of ranks of elliptic curves has become a central topic in number theory. It remains nevertheless mysterious. For example, it is not known if the rank of elliptic curves over a given number field is bounded. On the other hand, over function fields, the rank is known to be unbounded due Ulmer \cite{Ulmer-largerank}. To get a better understanding of ranks of elliptic curves in general, one can hope to determine the average rank of elliptic curves. The average rank for the family of elliptic curves over functions fields was first bounded by Brumer \cite{Brumer}, who gave an upper bound of $2.3$. In this paper, we obtain an upper bound less than $2$, which particularly shows that a positive proportion of elliptic curves has either rank $0$ or $1$. 

Let $\CC = \PP^{1}/\BF_{q}$ where $q \geq 5$ is a prime. Let $\BF_{q}(t)=\BF_q(\CC)$ be the function field of the curve $\mathcal{C}$. Each elliptic curve over $\BF_{q}(t)$ has a model
$$
E_{A,B} : y^2 = x^3 + Ax +B 
$$
where $A,B$ are in $\BF_{q}[t]$ satisfying $-16(4A^3 + 27B^2)\neq 0$. 
Define the canonical (naive) \textit{height} of the elliptic curve as follows; let $E'$ be any elliptic curve isomorphic to $E$ of the form $ E'_{C,D} : y^2= x^3+Cx+D$ where $C,D \in \BF_{q}[t],$ then
$$
h(E) := \inf_{E' \cong E}\Big(\max \{3\deg C, 2\deg D\} \Big).
$$
For ease of notation, we introduce 
$$
\mathcal{D}(d) = \{ E_{A,B} : \deg A =  \big \lfloor \frac{d}{3} \big \rfloor,\ \deg B = \big \lfloor \frac{d}{2} \big \rfloor  \}. 
$$
Write $\rk(E)$ for the Mordell-Weil rank of $E(\BF_{q}(t))$.
Our main theorem is the following. 
\begin{thm}
    Let $q \geq 5$ be a prime. Then, 
    $$
        \limsup_{d \rightarrow \infty} \frac{1}{\#\mathcal{D}(d)} \sum_{E \in \mathcal{D}(d)} \rk(E) \leq  \frac{25}{14}.
    $$
\end{thm}
With a little bit more effort one can relate our set $\mathcal{D}(d)$ to the more customary, 
$$
\{ 0 < \deg A \leq \big \lfloor \frac{d}{3} \big \rfloor , \ 0 < \deg B \leq \big \lfloor \frac{d}{2} \big \rfloor \}.
$$
Notice that for $d$ growing to infinity, the relative density of the two sets differ only by $\leq C / q$ with $C$ an absolute constant. With a little bit more effort one can also restrict the summation to minimal models. We avoid these restrictions as they make the proof somewhat more complicated, but for little gain, in our view. 

Brumer in \cite{Brumer} also gave an upper bound of $2.3$ for the average rank of elliptic curves over $\mathbb{Q}$  under the assumption of the General Riemann Hypothesis for elliptic curve $L$-functions. Under the same conditions, his work was later improved to $2$ by Heath-Brown \cite{HBrown}, and then to $25/14 \approx 1.8$ by Young in \cite{Matthew}. The author is inspired by the work of Young \cite{Matthew} and proves the analog of his result in the context of the function field. In the algebraic spectrum of similar studies, Bhargava and Shankar \cite{Bha-Shankar} remarkably proved an unconditional upper bound of $1.5$ for the average rank of elliptic curves over $\mathbb{Q}$. 

\section{Notation and Background}

Let $\BF_{q}$ be a finite field whose cardinality $q$ is a prime with $q \geq 5.$ We denote by $\CH$ the set of polynomials in $\BF_{q}[t]$ and by $\CM$ the set of monic polynomials in $\BF_{q}[t]$. And by $\CH_{n}$ and $\CM_{n}$ the set of polynomials of degree $n$ and the set of monic polynomials of degree $n$ in $\BF_{q}[t],$ respectively. The norm of a nonzero polynomial $F \in \BF_{q}[t]$ is defined as $|F| = q^{\deg F}$ and $|F|=0$ for $F=0$.


\subsection{$L$-functions attached to Dirichlet Characters} Let $h \in \BF_{q}[t]$ be a monic polynomial. Then a \textit{Dirichlet character} of modulus $h$ is a homomorphism 
$$
\chi : ( \BF_{q}[t]/ h \BF_{q}[t])^{\times} \rightarrow \mathbb{C}^{\times}.
$$
For any multiple $m h$ of $h,$ $\chi$ induces a homomorphism $(\BF_{q}[t]/ m h \BF_{q}[t])^{\times} \rightarrow \mathbb{C}^{\times}$
We call a character $\chi$ of modulus $h$ primitive if it cannot be induced from a modulus of lower degree and refer to $h$ as the conductor of $\chi$. We can evaluate a Dirichlet character $\chi$ of
conductor $h$ at an element $g \in \BF_{q}[t]$ by
\[
\chi(g) = \begin{cases} 
\chi ( g \Mod{h}) & \text{if}\ h\ \text{and}\ g\ \text{are coprime,}\\
0 & \text{else.}
\end{cases}
\]
The principal Dirichlet character $\chi_{0}$ of modulus $h$ is defined by the property that $\chi_{0}(g) = 1 $ if $h, g$ are coprime and $\chi_{0}(g)= 0$ otherwise. 

To any Dirichlet character $\chi$, we attach a Dirichlet $L$-function of a complex variable $u$ by
\[
\CL(u,\chi) = \sum_{f \in \CM} \chi(f) u^{\deg f}. 
\]
The $L$-function has Euler product
\begin{align*}
    \CL(u,\chi) = \prod_{P} (1-\chi(P)u^{\deg P})^{-1}
\end{align*}
where the product runs over all monic prime polynomials in $\BF_{q}[t]$. Let $n \geq \deg h$ be an integer. If $\chi$ is non-principal, it is easy to show that 
\[
\sum_{f \in \CM_{n}} \chi(f) = 0 
\]
and thus $\CL(u, \chi)$ is in fact a polynomial of degree at most $\deg h -1.$ If $\chi(ag) = \chi(g)$ for all $a \in \BF_{q}^{\times}$ and $g \in \BF_{q}^{\times},$ then $\chi$ is said to be \textit{even}.

\subsection{Quadratic characters}
The quadratic character $\Big ( \frac{\cdot}{P}\Big)$ for any prime $P \in \BF_{q}[t]$ is defined by
\[
\Big (\frac{f}{P} \Big )= 
\begin{cases} 
1 & \text{if}\ f\ \text{is a non-zero square} \bmod P \\
-1 & \text{if}\ f\ \text{is a non-square} \bmod P  \\
0 & \text{if}\ P | f  
\end{cases}
\]

The above definition can be multiplicatively extended to $\Big (\frac{\cdot}{D}\Big)$ for any non-zero monic $D \in \BF_{q}[t]$ and we write $\chi_{D}$ for the quadratic character $\Big ( \frac{\cdot}{D} \Big)$. The reciprocity law asserts that if $A, B \in \BF_{q}[t]$ are relatively prime, non-zero monic polynomials, then 
$$
\Big ( \frac{A}{B} \Big ) = (-1)^{\frac{q-1}{2}\deg A \deg B} \Big ( \frac{B}{A} \Big ). 
$$
This relation continues to hold if $A,B$ are not coprime as both sides vanish. 

\subsection{Exponential Function} We recall the exponential function introduced by Hayes in \cite{Hayes1966}. Note that the elements of $\BF_q((\frac{1}{t}))$ can be expressed as Laurent series. Specifically, each $a \in \BF_{q}((\frac{1}{t}))$ can be written uniquely as
\begin{align}\label{expansion}
a = \sum\limits_{i \in \Z} a_i \Big (\frac{1}{t}\Big)^i
\end{align}
where $a_i \in \BF_q$ such that all but finitely many $a_i$ with $i \geq 0$ are non-zero. There is a valuation defined by
$$
\nu(a) = \text{smallest}\ i \ \text{such that}\ a_i \neq 0.
$$
For $a \in \BF_{q}((\frac{1}{t})),$ the exponential function is given by
$$
e(a)= e^{2\pi i a_1/q}
$$
where $a_1$ is the coefficient of $\frac{1}{t}$ in \eqref{expansion}. We recall from \cite{Hayes1966} that $e(a+b)=e(a)e(b)$ for $a,b \in \BF_{q}((\frac{1}{t})).$ For $A \in \F_q[t]$ we have $e(A)=1.$ If $A,B,H \in \BF_q[t]$ are such that $A \equiv B \pmod H,$ then $e(A/H)= e(B/H).$

\subsection{Gauss sum}
The generalized Gauss sum is defined by
$$
G(u,\chi)= \sum\limits_{V \Mod f} \chi(V) e \Big (\frac{uV}{f}\Big ).
$$
For our later work, we require knowledge of $G(u,\chi_P)$ for a non-zero monic prime $P \in \BF_{q}[t]$. The next lemma allows us to compute $G(u,\chi_P)$ which is due Gauss. 
\begin{lem}\label{GaussLemma}
    Let $P \in \BF_{q}[t]$ a monic prime polynomial. Then
    \begin{align*}
        G(V, \chi_{P}) = |P|^{1/2}\Big(\frac{V}{P}\Big ) .
    \end{align*}
\end{lem}
\begin{proof}
   It follows from \cite[Lemma 3.2]{AFlo}.
\end{proof}

\section{Poisson Summation Formula}
For $F$ a general periodic function$\pmod f$, define the Fourier transform of $F$ as
$$
\widehat{F}(u; f)=  \sum\limits_{V\Mod f} F(V) e \Big (\frac{uV}{f} \Big ).
$$
The following Poisson summation formula holds. 
\begin{lem}\label{possionformula}
Let $F$ be a periodic function$\pmod f$ such that $f \in \BF_{q}[t]$ is a polynomial of degree $n$ and let $m$ be a positive integer. Then
    \begin{align}\label{poisson}
    \sum\limits_{g \in \CM_{m}}  F(g) = \frac{q^m}{|f|} \sum\limits_{\deg V \leq n-m-1} \widehat{F}(V;f)e \Big (\frac{-Vx^m}{f} \Big).       
    \end{align}
If $V=0$ then $\deg V = -\infty$ by convention and so the above equality holds. 
\end{lem}

\begin{proof}
    For any polynomial $g \in \BF_{q}[t]$, we have the following equality using the definition of the Fourier transform, 
\begin{align}\label{observation}
\frac{1}{|f|} \sum\limits_{V\Mod f} \widehat{F}(-V;f)e \Big (\frac{Vg}{f} \Big) = \frac{1}{|f|}\sum\limits_{u\Mod f} F(u) \sum\limits_{V\Mod f} e\Big ( \frac{V(g-u)}{f} \Big).
\end{align}
If $ u \neq g,$ then $\sum\limits_{V\Mod f} e\Big ( \frac{V(g-u)}{f} \Big) = 0,$ since we can find a polynomial $h$ such that $e\Big ( \frac{h(g-u)}{f} \Big) \neq 1$ and then
$$
e\Big ( \frac{h(g-u)}{f} \Big )\sum\limits_{V\Mod f} e\Big ( \frac{V(g-u)}{f} \Big) = \sum\limits_{V\Mod f} e \Big ( \frac{(V+h)(g-u)}{f} \Big ) = \sum\limits_{V\Mod f} e \Big( \frac{V(g-u)}{f}\Big ).
$$
Hence the only non-zero term on the right hand side of \eqref{observation} is determined by $u=g.$ Therefore, we obtain
\begin{align}\label{firstfact}
F(g) = \frac{1}{|f|} \sum\limits_{V\Mod f} \widehat{F}(-V;f)e \Big (\frac{Vg}{f} \Big).
\end{align}

Each $g \in \CM_{m}$ can be written as $g = x^m + u $ with $\deg u \leq m-1.$ Using \eqref{firstfact}, we have 
\begin{align*}
 \sum\limits_{g \in \CM_{m}} F(g) = \frac{1}{|f|} \sum\limits_{V\Mod f} \widehat{F}(-V;f) e \Big ( \frac{Vx^m}{f} \Big) \sum\limits_{\deg u \leq m-1} e\Big( \frac{Vu}{f} \Big) = \CS_1 + \CS_2
 \end{align*}
where 
\begin{align*}
 \CS_1 &= \frac{1}{|f|} \sum\limits_{\deg V \leq n-m-1} \widehat{F}(-V;f) e \Big ( \frac{Vx^m}{f} \Big) \sum\limits_{\deg u \leq m-1} e\Big(\frac{Vu}{f} \Big), \\
 \CS_2 &= \frac{1}{|f|} \sum\limits_{n-m \leq \deg V \leq n-1} \widehat{F}(-V;f) e \Big ( \frac{Vx^m}{f} \Big) \sum\limits_{\deg u \leq m-1} e\Big(\frac{Vu}{f} \Big).
\end{align*}
We first evaluate $\CS_1$. When $\deg V \leq n-m-1$ and $\deg u \leq m-1$, then $e \Big( \frac{Vu}{f} \Big ) = 1.$ Since there are $q^m$ polynomials of degree $ \leq m-1$, it follows under $V \mapsto -V$ that
\begin{align}\label{firstsummand}
\CS_1 = \frac{q^m}{|f|} \sum\limits_{\deg V \leq n-m-1} \widehat{F}(V;f) e \Big ( \frac{-Vx^m}{f} \Big ).
\end{align}
Now we will show that $\CS_2=0.$ We re-write
$$
\sum\limits_{\deg u \leq m-1} e \Big(\frac{Vu}{f} \Big) = 1 + \sum\limits_{i=0}^{m-1}\sum\limits_{c=1}^{q-1}\sum\limits_{u \in c\CM_i} e\Big(\frac{Vu}{f}\Big).
$$
Notice that if $i \leq n-2-\deg V$, then $e \Big (\frac{Vu}{f} \Big )= 1.$ If $ i \geq n - \deg V$, then it follows from \cite[Lemma 3.7]{Hayes1966} that $\sum\limits_{u \in c\CM_i} e\Big( \frac{Vu}{f} \Big ) = 0. $ Since $n - 1 - \deg V \leq m - 1$ in $\CS_2,$ combining all these above
\begin{align*}
\sum\limits_{\deg u \leq m-1} e \Big(\frac{Vu}{f} \Big) &= 1 + \sum\limits_{i=0}^{n-2-\deg V} \sum\limits_{c=1}^{q-1} q^i + \sum\limits_{c=1}^{q-1} 
\sum\limits_{u \in c\CM_{n-1-\deg V}} e\Big (\frac{Vu}{f} \Big ) \\
&= 1 + (q-1) \frac{q^{n-1-\deg V} - 1}{q-1} - q^{n-1-\deg V} \\
&= 0
\end{align*}
Hence $\CS_2 = 0$. Combining this with $\eqref{firstfact}$ and $\eqref{firstsummand},$ completes the proof of $\eqref{poisson}$.
\end{proof}

\begin{cor}\label{non-monicpoisson}
     Let $F$ be a periodic function$\pmod f$ such that $f \in \BF_{q}[t]$ is a polynomial of degree $n$ and let $m$ be a positive integer. Then
    \begin{align}\label{poisson2}
    \sum\limits_{g \in \CH_{m}}  F(g) = \frac{q^m}{|f|} \Big (\sum\limits_{\deg V = n-m-1} -\widehat{F}(V;f) + \sum\limits_{\deg V < n-m-1} (q-1)\widehat{F}(V;f) \Big ).       
    \end{align}
If $V=0$ then $\deg V = -\infty$ by convention and so the above equality holds. 
\end{cor}

\begin{proof}
Each $g \in \CH_{m}$ can be written as $g = ax^m + u $ where $a \in \BF_{q}^{\times}$ and $\deg u \leq m-1.$ Using \eqref{firstfact} above, 
\begin{align*}
 \sum\limits_{g \in \CH_{m}} F(g) &= \frac{1}{|f|} \sum_{a \in \BF_{q}^{*}} \Big ( \sum\limits_{V\Mod f} \widehat{F}(-V;f) e \Big ( \frac{Vax^m}{f} \Big) \Big ) \sum\limits_{\deg u \leq m-1} e\Big( \frac{Vu}{f} \Big) \\
 &= \CS_1 + \CS_2
 \end{align*}
where 
\begin{align*}
 \CS_1 &= \frac{1}{|f|} \sum_{a \in \BF_{q}^{*}} \Big ( \sum\limits_{\deg V \leq n-m-1} \widehat{F}(-V;f) e \Big ( \frac{Vax^m}{f} \Big) \Big )\sum\limits_{\deg u \leq m-1} e\Big(\frac{Vu}{f} \Big), \\
 \CS_2 &= \frac{1}{|f|} \sum_{a \in \BF_{q}^{*}} \Big (\sum\limits_{n-m \leq \deg V \leq n-1} \widehat{F}(-V;f) e \Big ( \frac{Vax^m}{f} \Big) \Big ) \sum\limits_{\deg u \leq m-1} e\Big(\frac{Vu}{f} \Big).
\end{align*}
With an argument similar to the proof of Lemma \ref{poisson}, we see that $\CS_2 = 0$. It remains to evaluate $\CS_1$. When $\deg V \leq n-m-1$ and $\deg u \leq m-1$, then $e \Big( \frac{Vu}{f} \Big ) = 1.$ Since there are $q^m$ polynomials of degree $ \leq m-1$, it follows under $V \mapsto -V$ that
\begin{align}\label{firstsummand2}
\CS_1 = \frac{q^m}{|f|} \sum_{a \in \BF_{q}^{*}} \Big (\sum\limits_{\deg V \leq n-m-1} \widehat{F}(V;f) e \Big ( \frac{-Vax^m}{f} \Big ) \Big ).
\end{align}
Notice that if $\deg V < n-m-1$, then $e \Big ( \frac{-Vax^m}{f} \Big ) = 1$. If $\deg V = n-m-1$, then $\sum\limits_{a \in \BF_{q}^{*}} e \Big ( \frac{-Vax^m}{f} \Big ) = -1$ and hence
$$
   \CS_1 = \frac{q^m}{|f|} \Big (\sum\limits_{\deg V = n-m-1} -\widehat{F}(V;f) + \sum\limits_{\deg V < n-m-1} (q-1)\widehat{F}(V;f) \Big ).
$$
which completes the proof. 
\end{proof}

\section{The Main Estimate}
For a non-principal character $\chi$ with conductor $h$, the Dirichlet $L$-function associated to $\chi$ is defined by 
\[
\CL(u,\chi) = \sum_{f \in \CM} \chi(f) u^{\deg f} = \prod_{P} (1-\chi(P)u^{\deg P})^{-1}
\]
where the summation is over all monic polynomials in $\BF_{q}[t]$ and the product is over all monic prime polynomials in $\BF_{q}[t]$. We will need the following result to bound the size of $L$-functions. Note that the result holds for characters of any order. 

\begin{lem}\label{boundingL}
    Let $\chi$ be a primitive character of conductor $h$ defined over $\BF_{q}[t]$. Then, for $|u| \leq \frac{1}{\sqrt{q}}$ and for all $\varepsilon > 0$, 
    $$
| \CL(u, \chi) | \ll_{\varepsilon} q^{\varepsilon \deg h}.
$$
In fact, if $\chi$ is a character of period $h$, then 
$$
| \CL(u, \chi) | \ll_{\varepsilon} ( 16 q^{\varepsilon})^{\deg h}.
$$
\end{lem}
\begin{proof}
    This is the Lindel{\"o}f hypothesis in function fields. See \cite[Theorem 20]{BUCUR_COSTA_DAVID_GUERREIRO_LOWRY–DUDA_2018} for the proof of the first part of the statement. 
    
    To prove the second part, assume that $\chi$ is a character of period $h$. Let $\chi = \chi_0 \chi_1 $ where $\chi_0$ is the principal character$\Mod h$ and $\chi_1$ is a primitive character$\Mod f$ with $f | h$. Since we have
    $$
\CL(u,\chi) = \prod_{P} ( 1- \chi(P)u^{\deg P})^{-1} \ \ \ \ \ \ \  \CL(u,\chi_1) = \prod_{P} (1 - \chi_1(P)u^{\deg P})^{-1}
    $$
and $\chi_0(P) = 1$ unless $P | h$, then $\CL(u,\chi)$ differs from $\CL(u,\chi_1)$ only on the primes $P$ with $P | h$. Therefore,
$$
\frac{\CL(u, \chi)}{\CL(u,\chi_1)} = \prod_{P | h} \frac{1-\chi_1(P)u^{\deg P}}{1  - \chi(P) u^{\deg P}}. 
$$
Applying the first part to $\chi_1$ allows us to set a bound
$$
| \CL(u,\chi) | = |\CL(u,\chi_1)| \Big |\frac{\CL(u, \chi)}{\CL(u,\chi_1)} \Big | \ll_{\varepsilon}  q^{\varepsilon \deg f} \Big ( 1 - \frac{1}{\sqrt{q}} \Big )^{- 2\deg h}
$$
Since $\deg f \leq \deg h$ and $(1 - 1/\sqrt{q})^{-1} \leq 4$, the proof is complete.
\end{proof}

\begin{lem}\label{sumofchi}
    Let $W \in \BF_q[t]$ be a polynomial of degree $l$ and $\chi$ a non-principal character$\Mod W$.  Then
    \begin{align*}
        \Big | \sum_{\substack{ V \in  \CH_{k} }} \chi(V) \Big | \ll_{\varepsilon} 16^l q^{\frac{k}{2} + 1  + \varepsilon l} \ \ \ \ \text{for any $\varepsilon > 0 $}.
    \end{align*}
\end{lem}
\begin{proof}
    For any $V \in \CH_k$, there is a unique monic polynomial $f\in \BF_q[t]$ and non-zero $a \in \BF_q$ such that $V=af$. Therefore 
    \begin{align}\label{writingmonic}
    \sum_{V \in \CH_k} \chi(V) = \sum_{a \neq 0} \chi(a) \sum_{V \in \CM_k} \chi(V).
    \end{align}    
    Using Perron's formula we have
    \begin{align*}
        \sum\limits_{ V \in \CM_{k}} \chi(V) = \frac{1}{2\pi i} \oint\limits_{|u|=r} \frac{\CL(u,\chi)}{u^k} \frac{\mathrm{d}u}{u} 
    \end{align*}
where we integrate along a circle with radius $r=q^{-1/2}$. By Lemma \ref{boundingL}, we have
\begin{align*}
    \Big | \frac{1}{2\pi i} \oint\limits_{|u|=r} \frac{\CL(u,\chi)}{u^k} \frac{\mathrm{d}u}{u} \Big| \ll_{\varepsilon} \frac{(16 q^{\varepsilon})^{l}}{2\pi} \oint\limits_{|u|=r} \Big | \frac{1}{u^k} \Big | \Big |\frac{\mathrm{d}u}{u} \Big | \ll_{\varepsilon} 16^{l} q^{\frac{k}{2} + \varepsilon l }
\end{align*}
Then it follows from \eqref{writingmonic} that
\begin{align*}
\Big | \sum_{\substack{ V \in  \CH_{k} }} \chi(V) \Big | \ll_{\varepsilon} 16^{l} q^{\frac{k}{2}+1 + \varepsilon l}
\end{align*}
for any $\varepsilon > 0$, as desired. 
\end{proof}

Before proceeding further, note that
$\CL(u, \chi)$ has a ``trivial'' zero at $u=1$ when $\chi$ is even. Thus
\[ \CL(u,\chi)= (1-u)^{\lambda}\CL^{*}(u,\chi), \ \ \ \lambda=
    \begin{dcases}
        0 & \text{if}\ \chi \ \text{is odd} \\
        1 & \text{if}\ \chi \ \text{is even} \\
    \end{dcases}
\]
where $\CL^{*}(u,\chi)$  is a polynomial of degree 
$\delta = \deg h - 1 - \lambda$. 
\begin{lem}\label{sumofchioverprimes}
    Given $W \in \BF_{q}[t]$ of degree $l$ and a non-principal character $\chi \Mod{W}$, we have
    \begin{align*}
        \Big | \sum_{\substack{P \in \CM_{n} \\ P \ \text{prime} }} \frac{\chi(P) }{|P|^{1/2}} \Big |  \ll_{\varepsilon} q^{\varepsilon l} \ \ \ \ \ \text{for any $\varepsilon > 0$}.
    \end{align*}
\end{lem}

\begin{proof}
    The Riemann Hypothesis for curves over function fields \cite{Weil} asserts that all zeroes of $\CL^{*}(u,\chi)$ are on the circle $|u|=q^{-1/2}$. 
    Thus, we may express
    $$
    \CL^{*}(u,\chi) = \det ( 1-u\sqrt{q} \Theta )
    $$
for a $\delta \times \delta$ unitary matrix  $\Theta$. By taking a logarithmic derivative of the identity
$$
\det ( 1-u\sqrt{q} \Theta ) = (1-u)^{-\lambda} \prod_{P} (1-\chi(P)u^{\deg P})^{-1}
$$
which comes from writing $\CL^{*}(u,\chi) = (1-u)^{-\lambda}\CL(u,\chi)$, we find
$$
-\tr(\Theta^n) = \frac{\lambda}{q^{n/2}} + \frac{1}{q^{n/2}} \sum\limits_{f \in \CH_{n}} \Lambda(f)\chi(f)
$$
where $\Lambda$ is the von Mangoldt function. It then follows that
$$
 \Big | \sum_{\substack{P \in \CM_{n} \\ P \ \text{prime} }} \frac{\chi(P) }{|P|^{1/2}} \Big | \ll \Big | \frac{\tr (\Theta^{n}) -  \lambda/q^{n/2} } {n} \Big | \ll \frac{\deg h}{n} \ll \frac{l}{n} \ll_{\varepsilon} q^{\varepsilon l}
$$
for any $\varepsilon > 0$, as desired.
\end{proof}

\begin{prop}\label{mainbound}
    Let $n,k,l$ be positive integers. Let $a \in \mathbb{F}_q^{\times}$. Then
    \begin{align*}
      \Big |  \sum_{\substack{P \in \CM_{n} \\ P \ \text{prime}}} \frac{1}{|P|^{1/2}} \Big (\sum_{\substack{V \in \mathcal{M}_k \\ \deg W =l}} - \Big (\frac{W}{P} \Big ) e\Big ( \frac{(a V)^3\overline{P}}{W^2} \Big ) \Big) \Big | \ll_\varepsilon q^{\varepsilon l } ( q^{k+2}+ 16^{l} q^{\frac{n}{2} + 2 + \frac{k}{2}} + 16^{l} q^{2l + 2+ \frac{k}{2}}  )     
    \end{align*} 
for any $\varepsilon>0$.
\end{prop}
\begin{rem}
Each non-monic $V$ of degree $k$ can be written as $a V'$ with $V'$ monic of degree $k$ and $a$ invertible. Therefore, summing the above Proposition over $a$ and using the triangle inequality establishes that
$$
\Big |  \sum_{\substack{P \in \CM_{n} \\ P \ \text{prime}}} \frac{1}{|P|^{1/2}} \Big (\sum_{\substack{\deg V =k \\ \deg W =l}} - \Big (\frac{W}{P} \Big ) e\Big ( \frac{V^3\overline{P}}{W^2} \Big ) \Big) \Big | \ll_\varepsilon q^{\varepsilon l } ( q^{k+2}+ 16^{l} q^{\frac{n}{2} + 2 + \frac{k}{2}} + 16^{l} q^{2l + 2+ \frac{k}{2}}  )
$$
We will not use this (more symmetric) result. 
\end{rem}

\begin{proof}
We first separate the variables by using the following expansion of multiplicative characters into additive characters via Gauss sums
\begin{align*}
    e \Big ( \frac{(aV)^3\overline{P}}{W^2} \Big ) = \frac{1}{\phi(W^2)} \sum\limits_{\chi \Mod{ W^2}} \chi((aV)^3\overline{P}) G(1,\overline{\chi})
\end{align*}
which is valid because of the fact; $(\overline{P}, W^2)=1$. Then we obtain
\begin{align*}
    \sum_{\substack{P \in \CM_{n} \\ P \ \text{prime} }} & \frac{1}{|P|^{1/2}} \Big (\sum_{\substack{V \in \mathcal{M}_k \\ \deg W =l}} - \Big (\frac{W}{P} \Big ) e\Big ( \frac{(aV)^3\overline{P}}{W^2} \Big ) \Big) \\ & = \sum_{\substack{P \in \CM_{n} \\ P \ \text{prime}}} \frac{1}{|P|^{1/2}} \Big ( \sum_{\substack{V \in \mathcal{M}_k \\ \deg W =l}} - \Big (\frac{W}{P} \Big ) \frac{1}{\phi(W^2)} \sum\limits_{\chi \Mod{ W^2}} \chi((aV)^3\overline{P}) G(1,\overline{\chi}) \Big ) \\
    & = \sum_{\substack{P \in \CM_{n} \\ P \ \text{prime}}} \sum_{\substack{\chi \Mod{ W^2}}} \frac{-\chi(\overline{P}) \Big (\frac{W}{P} \Big )}{|P|^{1/2}} \sum_{\substack{V \in \mathcal{M}_k \\ \deg W =l}} \frac{1}{\phi(W^2)} \chi^3(a V) G(1,\overline{\chi})
\end{align*}
Now we may split the following sum
\begin{align*}
\sum_{\substack{\chi \Mod{ W^2}}} \sum_{\substack{ P \in \CM_{n} \\ P \ \text{prime}}} \frac{-\chi(\overline{P}) \Big (\frac{W}{P} \Big )}{|P|^{1/2}} \sum_{\substack{V \in \mathcal{M}_k \\ \deg W =l}} \frac{1}{\phi(W^2)} \chi^3(a V) G(1,\overline{\chi})
\end{align*}
into three cases.
\textit{Case $1$ :} $\chi^{3}=\chi_{\mathrm{o}}$ and $\chi(\overline{P}) \Big (\frac{W}{P} \Big ) \neq \chi_{\mathrm{o}}$. Then we have
\begin{align}\label{case1-1stbound}
    \Big | \sum_{\substack{V \in \mathcal{M}_k}} \chi^3(V) \Big | \ll q^{k + 1}
\end{align}
and applying Lemma \ref{sumofchioverprimes} to $-\overline{\chi}.\chi_{W}$,
\begin{align}\label{case1-2ndbound}
     \Big | \sum_{\substack{P \in \CM_{n} \\ P \ \text{prime}}} \frac{-\chi(\overline{P}) \Big (\frac{W}{P} \Big )}{|P|^{1/2}} \Big | \ll_{\varepsilon} q^{\varepsilon l} \ \ \ \ \ \text{for any $\varepsilon > 0$}.
\end{align}
Using \eqref{case1-1stbound} and \eqref{case1-2ndbound}
\begin{align*}
 \Big | \sum_{\substack{P \in \CM_{n} \\ P \ \text{prime}}} \sum_{\substack{\chi \Mod{ W^2} \\ \chi^3 = \chi_{\mathrm{o}}}} \frac{-\chi(\overline{P}) \Big (\frac{W}{P} \Big )}{|P|^{1/2}} \sum_{\substack{V \in \mathcal{M}_k \\ \deg W =l}} \frac{1}{\phi(W^2)} \chi^3(a V) G(1,\overline{\chi}) \Big | & \ll_{\varepsilon}  \sum_{\substack{\chi \Mod{ W^2} \\ \chi^3 = \chi_{\mathrm{o}}}} q^{k+1+\varepsilon l } \sum_{\substack{\deg W =l}}\frac{|G(1,\overline{\chi})|}{\phi(W^2)} \\
& \ll_{\varepsilon} q^{k+1+\varepsilon l } \sum_{\substack{\deg W =l}} \frac{q^l}{\phi(W^2)} \\
& \ll_{\varepsilon} q^{k+2+\varepsilon l}.
\end{align*}

\textit{Case $2$ :} $\chi^{3} \neq \chi_{\mathrm{o}}$ and $\chi(\overline{P}) \Big (\frac{W}{P} \Big ) = \chi_{\mathrm{o}}$. Applying Lemma \ref{sumofchi} to $\chi^3$
\begin{align}\label{case2-1stbound}
    \Big | \sum_{\substack{V \in \mathcal{M}_k}} \chi^3(V) \Big | \ll_{\varepsilon} 16^l q^{\frac{k}{2} + 1+ \varepsilon l} \ \ \ \ \text{for any $\varepsilon > 0 $}
\end{align}
and since $\chi(\overline{P}) \Big (\frac{W}{P} \Big ) = \chi_{\mathrm{o}}$, we have $\chi(P)= \Big (\frac{W}{P}\Big )$. So,
\begin{align*}
\Big |\sum_{\substack{P \in \CM_{n} \\ P \ \text{prime}}} \sum_{\substack{\chi \Mod{ W^2} \\ \chi(P) = \big (\frac{W}{P} \big )}} \frac{-\chi(\overline{P}) \Big (\frac{W}{P} \Big )}{|P|^{1/2}} \sum_{\substack{V \in \mathcal{M}_k \\ \deg W =l}} \frac{1}{\phi(W^2)} \chi^3(V) G(1,\overline{\chi}) \Big | & \ll_{\varepsilon} \sum_{\substack{P \in \CM_{n} \\ P \ \text{prime} }} \frac{16^l q^{\frac{k}{2}+ 1+ \varepsilon l }}{|P|^{1/2}}\sum_{\substack{\deg W =l}} \frac{|G(1,\overline{\chi})|}{\phi(W^2)}  \\
& \ll_{\varepsilon} 16^l q^{\frac{n}{2} + 2 + \frac{k}{2}+ \varepsilon l}.
\end{align*}

\textit{Case $3$:} $\chi^{3} \neq \chi_{\mathrm{o}}$ and $\chi(\overline{P}) \Big (\frac{W}{P} \Big ) \neq \chi_{\mathrm{o}}$. Applying Lemma \ref{sumofchi} to $\chi^3$, 
\begin{align}\label{case3-1stbound}
\Big | \sum_{\substack{V \in \mathcal{M}_k}} \chi^3(V) \Big | \ll_{\varepsilon} 16^l q^{\frac{k}{2} + 1 +\varepsilon l } \ \ \ \ \text{for any $\varepsilon > 0 $}.
\end{align}
and applying Lemma \ref{sumofchioverprimes} to $-\overline{\chi}.\chi_{W}$,
\begin{align}\label{case-32ndbound}
    \Big | \sum_{\substack{P \in \CM_{n} \\ P \ \text{prime} }} \frac{-\chi(\overline{P}) \Big (\frac{W}{P} \Big )}{|P|^{1/2}} \Big | \ll_{\varepsilon} q^{\varepsilon l} \ \ \ \ \ \text{for any $\varepsilon > 0$}.
\end{align}
Using \eqref{case3-1stbound} and \eqref{case-32ndbound} gives 
\begin{align*}
\Big | \sum_{\substack{P \in \CM_{n} \\ P \ \text{prime} }}  \sum_{\substack{\chi \Mod{ W^2}}} \frac{-\chi(\overline{P}) \Big (\frac{W}{P} \Big )}{|P|^{1/2}}  &\sum_{\substack{V \in \mathcal{M}_k \\ \deg W =l}} \frac{1}{\phi(W^2)} \chi^3(V) G(1,\overline{\chi}) \Big | \\ & \ll_{\varepsilon}  \sum_{\substack{\chi \Mod{ W^2}}} 16^{l} q^{\frac{k}{2}+ 1+\varepsilon l +\varepsilon l } \sum_{\substack{\deg W =l}}\frac{|G(1,\overline{\chi})|}{\phi(W^2)}
\\
& \ll_{\varepsilon} 16^{l} q^{\frac{k}{2}+ 1+ \varepsilon l + \varepsilon l } \sum_{\substack{\deg W =l}} q^l
\\
& \ll_{\varepsilon} 16^{l} q^{2l+2+ \frac{k}{2}+ \varepsilon l +\varepsilon l}.
\end{align*}

\end{proof}

\section{Averaging rank over function fields}

\subsection{$\mathcal{L}$-function associated to elliptic curves} 

The $L$-function of $E/\BF_{q}(t)$ is defined analogously as for $E/\Q.$ More precisely, let $P$ be a prime of $\BF_{q}[t]$ i.e. $P \in \BF_{q}[t]$ is a monic irreducible polynomial or $P=P_{\infty}.$  If $P$ is a prime of good reduction, then the reduction of $E$ (which we also denote by $E$) is an elliptic curve over $\BF_{P}=\BF_{q}[t]/(P) \simeq \BF_{q^{\text{deg} P}}$ (where $\BF_{\infty}=\BF_{q}$ since the prime at infinity has degree 1), and
Let $N_P$ be the number of points in the projective coordinates over $\BF_{P}$, i.e. 
$$
N_P = | \{ \mathcal{O} \} \cup \{(x,y) \in \BF_{P}^2 : y^2 - (x^3 + Ax + B) \equiv 0\bmod P\} |.
$$
Let $a_P = q^{\deg P} + 1 - N_{P}$ and then 
$$
\#E(\BF_{P})= q^{\deg P} + 1 - a_P, \ a_P = \alpha_P + \overline{\alpha}_P, \ |\alpha_P |= \sqrt{q^{\deg P}}.
$$
Let 
$$
\mathcal{L}_P (E, u)= 1- a_P u + q^{\deg P} u^2= (1-\alpha_P u) (1 - \overline{\alpha}_P u)
$$
be the $L$-function of $E/\BF_{P}.$ If $P$ is a prime of bad reduction, we define
$$
\mathcal{L}_P(E,u)= (1-a_P u)
$$
where $a_P= 0, 1, -1$ depending on the type of bad reduction (additive, split multiplicative and non-split multiplicative respectively). Let $N_E$ be the conductor of $E$ which is the product of the primes of bad reduction with the appropriate powers. Let $M_E$ (respectively $A_E$) be the product of multiplicative (respectively additive) primes of $E$.
Then $N_E = M_E A_E^2$.

The $L$-\textit{function} of $E$ is
$$
\mathcal{L}(E,u)= \prod_{P \nmid N_E} \frac{1}{\mathcal{L}_P(E,u^{\deg P})} \prod_{P | N_E} \frac{1}{\mathcal{L}_P(E,u^{\deg P})}.
$$
It is proven by Weil that $\mathcal{L}(E,u)$ is a polynomial of degree $ N= \deg N_E - 4= \deg (M_{E}) + 2\deg(A_E) -4,$ setting $u=q^{-s}$ it satisfies a functional equation $s \xleftrightarrow{} 2-s,$ and its zeros lie on Re $s=1$ for elliptic curves in $\mathcal{D}(d)$ for any $d$. More precisely,
$$
\CL(E,u)= \CL(E,q^{-s})= \prod_{i=1}^{N} (1-\mu_i q^{-s}) 
$$
where each $\mu_i$ is an algebraic integer of absolute value $q$ in every complex embedding, moreover $|\mu_{i}| = q$.

In order to prove the main proposition, we need auxiliary lemmas. 

\begin{lem}\label{lemma1} Let $N = \deg N_E - 4$. We have
    \begin{align}\label{nthcoefficiets}
\frac{1}{n}\sum\limits_{i=1}^{N} \Big ( \frac{\mu_i}{q} \Big )^n = - \frac{1}{q^n} \sum\limits_{k |n} \Big (\sum_{\substack{P \nmid N_E \\ \deg P = n/k}} \frac{\alpha_{P}^k + \overline{\alpha}_{P}^k}{k} + \sum_{\substack{P | N_E \\ \deg P = n/k}} \frac{a_{P}^k}{k} \Big ).
\end{align}
\end{lem}

\begin{proof}
    For any given elliptic curve $E$, we have the following equality
$$ 
\prod_{i=1}^{N} (1-\mu_i u) = \prod_{P \nmid N_E} \frac{1}{\mathcal{L}_P(E,u^{\deg P})} \prod_{P | N_E} \frac{1}{\mathcal{L}_P(E,u^{\deg P})}.
$$
By change of variable $u=\frac{z}{q}$ with $|z| < 1$ and taking the $\log$ of both sides, 
\begin{align*}
\sum\limits_{i=1}^{N} &\log \Big (1- \frac{\mu_i}{q} z \Big ) = \sum_{P \nmid N_E} -\log \Big (\CL_{P} \Big (E,\Big ( \frac{z}{q} \Big )^{\deg P} \Big ) \Big) + \sum_{P | N_E} - \log \Big (\CL_{P} \Big ( E, \Big( \frac{z}{q} \Big) ^{\deg P} \Big) \Big ) \\
&= \sum_{P \nmid N_E} \Big ( -\log \Big (1-\frac{\alpha_{P}}{q^{\deg P}} z^{\deg P} \Big ) - \log \Big ( 1-\frac{\overline{\alpha}_{P}}{q^{\deg P}} z^{\deg P} \Big) \Big) + \sum_{P | N_E} -\log \Big ( 1- \frac{a_P}{q^{\deg P}} z^{\deg P} \Big )
\end{align*}

Taylor expanding in $z$ yields
$$
- \sum\limits_{n=1}^{\infty} \sum\limits_{i=1}^{N} \frac{(\frac{\mu_i}{q} z)^n}{n} =   \sum\limits_{n=1}^{\infty} \Big ( \sum_{P \nmid N_E} \frac{ \Big ( \frac{\alpha_{P}^n + \overline{\alpha}_{P}^n}{q^{n\deg P}} \Big ) z^{n\deg P}} {n} + \sum_{P | N_E} \frac{ (\frac{a_{P}^n}{q^{n\deg P}}) z^{n\deg P}}{n} \Big ). 
$$
Equating the $n$-th coefficients, we get the desired identity.
\end{proof}

Simplifying Lemma \ref{lemma1} we get the following.

\begin{lem}\label{bounding}
    \begin{align}\label{nthcoefficiets}
\frac{1}{n}\sum\limits_{i=1}^{N} \Big ( \frac{\mu_i}{q} \Big )^n =  \frac{1}{q^n} \sum_{\substack{P \nmid N_E \\ \deg P = n}} \Big (\alpha_P + \overline{\alpha_P} \Big ) + \frac{1}{2} + O(q^{-n / 100}).
\end{align}
\end{lem}
\begin{rem}
Note that this lemma implies that the sum over primes is bounded by $\deg N_E - 4.$
\end{rem}
\begin{proof}
This follows from the previous Lemma upon noticing that the sum 
$$
\frac{1}{q^n}\sum_{\substack{P \nmid N_E \\ \deg P = n}} (\alpha_P + \overline{\alpha_P})
$$
corresponds to terms with $k = 1$ in the previous Lemma (the contribution of the terms with $P | N_E$ is absorbed into $O(q^{-n/100})$). The terms with $k \geq 3$ contribute a total of $\ll q^{-n / 100}$. It remains to note that, the contribution of the terms with $k = 2$ is
$$
\frac{1}{q^n}\sum_{\substack{P \nmid N_E \\ \deg P = n / 2}} \frac{\alpha_P^2 + \overline{\alpha_P}^2}
{2} = - \frac{1}{2} + O(q^{-n / 100}).$$
This follows from \cite{MR4471523}[Lemma 2.5(3)] and \cite{MR4471523}[(2.7)] with $m=1$ and $k=2.$
\end{proof}

Let $A,B \in \BF_{q}[t]$. Define $F(A,B;P) = \sum\limits_{x \in \BF_{P}} \Big ( \frac{x^3+Ax +B}{P} \Big ),$ then  
   $$ 
   \widehat{F}(\alpha, \beta; P) =  \sum\limits_{A \Mod P} \sum\limits_{B \Mod P} F(A,B;P) e\Big (\frac{\alpha A}{P}\Big ) e\Big (\frac{\beta B}{P}\Big ).
   $$

\begin{lem}\label{matthewsformula}
    Let $P \in \BF_{q}[t]$ be a monic prime polynomial and $\overline{\beta}$ be defined by $\beta\overline{\beta} \equiv 1 \Mod P $ if $(\beta,P)=1$ and $\overline{0}=0$. Then we have
    $$
    \widehat{F}(\alpha, \beta; P) =  |P|^{3/2} \Big ( \frac{\beta}{P} \Big ) e\Big ( \frac{-\alpha^3\overline{\beta}^2}{P} \Big ). 
    $$
    Notice that if $P | \beta$ the sum is zero, consistent with the above formula.
\end{lem}

\begin{proof}
    By definition,
    \[
\widehat{F}(\alpha, \beta; P) = \sum\limits_{x \in \BF_{P}} \sum\limits_{A \Mod P} \sum\limits_{B \Mod P} \Big ( \frac{x^3+Ax +B}{P} \Big )  e\Big (\frac{\alpha A + \beta B}{P}\Big ).
    \]
    After the change of variables $B \mapsto B -x^3 - Ax$, we have
    \begin{align*}
\widehat{F}(\alpha, \beta; P) &= \sum\limits_{x \in \BF_{P}} e \Big (\frac{-x^3 \beta}{P} \Big) \sum\limits_{A \Mod P} e\Big ( \frac{A (\alpha - x \beta)}{P} \Big ) \sum\limits_{B \Mod P} \Big ( \frac{B}{P} \Big ) e \Big (\frac{\beta B}{P} \Big ) \\
&= \sum\limits_{x \in \BF_{P}} e \Big (\frac{-x^3 \beta}{P} \Big) \sum\limits_{A \Mod P} e\Big ( \frac{A (\alpha - x \beta)}{P} \Big ) G(\beta, \chi_{P}) \\
&= |P|^{3/2}\Big (\frac{\beta}{P} \Big) e\Big ( \frac{-\alpha^3\overline{\beta}^2}{P} \Big ) \ \ \   \text{by Lemma \ref{GaussLemma}. }
    \end{align*}
\end{proof}

We will now begin the proof of the following proposition. 

\begin{prop}\label{mainprop}
Let $\delta > 0$ be given. For $0 < |n| < (7 / 9 - \delta) d$ we have
\begin{align}\label{bound}
\sum_{E \in \mathcal{D}(d)} \frac{1}{q^n} \sum_{\substack{P \nmid N_E \\ \deg P = n}} \Big (\alpha_P + \overline{\alpha_P} \Big )  \ll_{\delta} q^{5 d / 6 - (\delta / 2) d}.
\end{align}
\end{prop}
\begin{proof}
Let $a = \lfloor d / 3 \rfloor$ and $b = \lfloor d / 2 \rfloor$.
For any $P$, $a_{P}$ is given by
$$
a_{P}= \alpha_{P} + \overline{\alpha}_{P} = \sum_{x \in \BF_{P}} \Big ( \frac{x^3+Ax+B}{P} \Big ). 
$$

Re-writing the expression in \eqref{bound} we can focus on obtaining a non-trivial upper bound for
\begin{align*}
\sum_{E \in \mathcal{D}(d)} \frac{1}{q^n} \sum_{\substack{P \nmid N_E \\ \deg P = n}} \sum_{x \in \BF_{P}}  \Big (\frac{x^3+Ax+B}{P} \Big ).
\end{align*}

This is equal to

\begin{align*}
     \sum_{\substack{P \nmid N_E \\ \deg P =  n}} \frac{1}{q^n} \sum_{x \in \BF_{P}} \sum_{ \deg A = a} \sum_{ \deg B = b } \Big ( \frac{x^3+Ax+B}{P} \Big ) & = \frac{1}{q^n}  \sum_{\substack{P \nmid N_E \\ \deg P = n}} \sum_{ \deg A = a} \sum_{ \deg B = b } F(A,B;P).    
    \end{align*}
    Applying Corollary \ref{non-monicpoisson} twice to the right-hand side above 
    \[
    \frac{1}{q^n} \Big (\sum_{\substack{P \nmid N_E \\ \deg P = n}} \frac{q^{a+b}}{|P|^2} \Big ( \sum_{\substack{\deg V = n-a -1 \\ \deg W = n-b-1}  } -\widehat{F}(V,W;P) + \sum_{\substack{ \deg V < n - a-1 \\ \deg W < n- b -1}} (q-1) \widehat{F}(V,W;P) \Big) \Big ). 
    \]
Using Lemma \ref{matthewsformula} yields
\begin{align*}
 \sum_{\substack{P \nmid N_E \\ \deg P = n}} \frac{q^{a+b -n }}{|P|^{1/2}} & \Big ( \sum_{\substack{\deg V = n-a -1 \\ \deg W = n-b-1}} - \Big (\frac{W}{P} \Big ) e\Big ( \frac{-V^3\overline{W}^2}{P} \Big ) + \sum_{\substack{\deg V < n- a -1 \\ \deg W = n- b -1}} - \Big (\frac{W}{P} \Big ) e\Big ( \frac{-V^3\overline{W}^2}{P} \Big ) \\
& + \sum_{\substack{ \deg V = n - a -1 \\ \deg W < n- b -1}} (q-1) \Big ( \frac{W}{P} \Big ) e\Big ( \frac{-V^3\overline{W}^2}{P} \Big ) + \sum_{\substack{ \deg V < n - a -1 \\ \deg W < n-b -1}} (q-1) \Big ( \frac{W}{P} \Big ) e\Big ( \frac{-V^3\overline{W}^2}{P} \Big ) \Big) \\
& =  \CS_1 + \CS_2 + \CS_3 + \CS_4 .
\end{align*}
Our goal is to bound each $|\CS_i|$ for $i =1,2,3,4.$ We first study $\CS_1.$ We may employ the following elementary reciprocity formula for function fields, which can be easily shown by using the Chinese remainder theorem;
\begin{align}\label{elementaryformula}
(\alpha\beta) \Big ( \frac{\overline{\alpha}}{\beta} + \frac{\overline{\beta}}{\alpha} \Big ) \equiv 1 \Mod{\alpha\beta}
\end{align}
where $\alpha,\beta \in \BF_{q}[t]$ are polynomials such that $(\alpha, \beta)=1, \alpha \overline{\alpha} \equiv 1 \Mod{\beta},$ and $\beta \overline{\beta} \equiv 1 \Mod{\alpha}.$ In our application, $\alpha = W^2$ and $\beta=P.$ Now $\CS_1$ has transformed into
$$
\CS_1 = \sum_{\substack{P \nmid N_E \\ \deg P = n}}
\frac{q^{a+b - n }}{|P|^{1/2}} \Big ( \sum_{\substack{\deg V = n- a -1 \\ \deg W = n- b -1}} - \Big (\frac{W}{P} \Big ) e\Big ( \frac{-V^3}{W^2P} \Big ) e\Big ( \frac{V^3\overline{P}}{W^2} \Big ) \Big ) $$
When $\deg V = n-a -1$ and $\deg W = n-b-1,$ then $e\Big ( \frac{-V^3}{W^2P} \Big ) = e^{2\pi i(-v_1)}$ where $v_1$ is the leading coefficient of $V$. It then follows from Proposition \ref{mainbound} that
\begin{align*}
\big | \CS_1 \big | &= \Big | \sum_{\substack{P \nmid N_E \\ \deg P = n}} \frac{q^{a+b- n }}{|P|^{1/2}} \Big ( \sum_{\substack{\deg V = n- a -1 \\ \deg W = n- b -1}} - \frac{1}{e^{2\pi i v_1}} \Big (\frac{W}{P} \Big ) e\Big ( \frac{V^3\overline{P}}{W^2} \Big ) \Big ) \Big | \\
& \ll_{\varepsilon} q^{a+b  + \varepsilon ( n- b - 1)} ( q^{- a +1}+ 16^{n-b -1} q^{-\frac{1}{2} a+\frac{3}{2}} + 16^{n- b -1} q^{\frac{3n}{2}-2 b - \frac{1}{2} a - \frac{1}{2}} )
\end{align*}
for any $\varepsilon >0$. Similarly, each $|\CS_i|$ for $i=2,3,4$ is bounded by the same upper bound (up to a constant). Hence, we obtain an estimate
\begin{align*}
\Big | \sum\limits_{E \in \mathcal{D}(d)} (\CS_1 + \CS_2 + \CS_3 + \CS_4)  \Big | 
\ll_{\varepsilon} q^{a+b + \varepsilon ( n- b - 1 )} ( q^{- a +1}+ 16^{n- b -1} q^{-\frac{1}{2}a + \frac{3}{2}} + 16^{n-b -1} q^{\frac{3n}{2}-2 b - \frac{1}{2} a - \frac{1}{2}} ).
\end{align*}

In conclusion, 
\begin{align*}
\frac{1}{ \# \mathcal{D}(d)} \Big |\sum\limits_{E \in \mathcal{D}(d)} (\CS_1 + \CS_2 + \CS_3 + \CS_4)\Big |
\ll_{\varepsilon} q^{\varepsilon ( n- b - 1)} ( q^{- a +1}+ 16^{n-b -1} q^{-\frac{1}{2}a + \frac{3}{2}} + 16^{n- b-1} q^{\frac{3n}{2}-2 b - \frac{1}{2} a - \frac{1}{2}} ).
\end{align*}
Recall that $|n| < (7/9 - \delta) d$. Taking $\varepsilon = \delta / 1000$ we obtain the claim.
\end{proof}

\begin{cor}\label{maincor} 
Let $0 < \delta < 1 / 1000$ be given. For $0 < |n| < (7 / 9 - \delta) d$ we have
    $$ \frac{1}{\# \mathcal{D}(d)} \sum_{E \in \mathcal{D}(d)} \frac{1}{n}\sum\limits_{i=1}^{N} \Big ( \frac{\mu_i}{q} \Big )^n = \frac{1}{2}  + O(q^{- (\delta / 2) d}).$$
\end{cor}
\begin{proof}
This follows from combining Lemma \ref{bounding} and Proposition \ref{mainprop} with the observation that 
$$
\#\mathcal{D}(d) \asymp q^{5 d / 6}.
$$
\end{proof}

\section{Prof of Main Theorem}

By work of Tate, 
$$
\frac{1}{\#\mathcal{D}(d)} \sum_{E \in \mathcal{D}(d)} \text{rk}(E) \leq \frac{1}{\# \mathcal{D}(d)} \sum_{E \in \mathcal{D}(d)} \text{ord}_{s = 1} \mathcal{L}(E, s)
$$
Let $$
T(z) = (1 / v) \sum_{|\ell| < v} (1 - |\ell / v|) z^{\ell}.
$$ 
This function has the property that $T(z) \geq 0$ for all complex $z$ with $|z| = 1$ and $T(1) = 1$. Therefore, the above is less than,
$$
\frac{1}{\#\mathcal{D}(d)} \sum_{E \in \mathcal{D}(d)} \sum_{i} T \Big ( \frac{\mu_i}{q} \Big )
$$
where $\mu_i$ are the zeros of the $L$-function associated with each $E \in \mathcal{D}(d)$.
Now appealing to the main Proposition of section $2$ we see that if $v < (7/9 - \delta) d$ then only the main term contributes. This gives a final bound, that is,
$$
\frac{1}{\# \mathcal{D}(d)} \sum_{E \in \mathcal{D}(d)} \frac{\deg N_E - 4}{(7 / 9 - \delta) d} + \frac{1}{2}.
$$
In the limit this is less than, 
$$
\frac{1}{7 / 9 - \delta} + \frac{1}{2}.
$$
Since $\delta > 0$ is arbitrary we can take $\delta$ to zero and then get the claim. 

\newpage
\printbibliography

\end{document}